\documentclass{amsart}
\usepackage{amsmath,amsthm,amssymb,amsfonts,enumerate,color}
\usepackage{mathrsfs}
\usepackage{amsmath,amstext,amsthm,amssymb,amsxtra}
\usepackage{txfonts} 
\usepackage[colorlinks, citecolor=blue,pagebackref,hypertexnames=false]{hyperref}
\allowdisplaybreaks
\usepackage{pgf,tikz}
\usepackage{relsize}

\begin{document}



	\newtheorem{theorem}{Theorem}[section]
	\newtheorem{proposition}[theorem]{Proposition}
	\newtheorem{corollary}[theorem]{Corollary}
	\newtheorem{lemma}[theorem]{Lemma}
	\newtheorem{definition}[theorem]{Definition}	
	\newtheorem{assum}{Assumption}[section]
	\newtheorem{example}[theorem]{Example}
	\newtheorem{remark}[theorem]{Remark}
	\newtheorem*{conjecture}{Conjecture}

\newcommand{\R}{\mathbb{R}}
\newcommand{\Rn}{{\mathbb{R}^n}}
\newcommand{\T}{\mathbb{T}}
\newcommand{\Tm}{{\mathbb{T}^m}}
\newcommand{\TmRn}{{\mathbb{T}^m\times\mathbb{R}^n}}
\newcommand{\Z}{\mathbb{Z}}
\newcommand{\Zm}{{\mathbb{Z}^m}}
\newcommand{\C}{\mathbb{C}}
\newcommand{\Q}{\mathbb{Q}}
\newcommand{\N}{\mathbb{N}}

	\newcommand{\supp}{{\rm supp}{\hspace{.05cm}}}
	\newcommand {\lb}{{\langle}}
	\newcommand {\rb}{\rangle}
 	\numberwithin{equation}{section}

\title[$L^p$ bounds for     Stein's   spherical  maximal    operators]{$L^p$ bounds for     Stein's    spherical  maximal    operators}

 \author[ N. Liu,  M. Shen, \ L. Song  and L.  Yan]{ Naijia Liu, \ Minxing Shen, \ Liang Song
	and \ Lixin Yan}

\address{Naijia Liu,
	Department of Mathematics,
	Sun Yat-sen University,
	Guangzhou, 510275,
	P.R.~China}
\email{liunj@mail2.sysu.edu.cn}	

\address{Minxing Shen,
	Department of Mathematics, Sun Yat-sen
	University, Guangzhou, 510275, P.R. China}
\email{shenmx3@mail2.sysu.edu.cn}	

\address{Liang Song,
	Department of Mathematics,
	Sun Yat-sen University,
	Guangzhou, 510275,
	P.R.~China}
\email{songl@mail.sysu.edu.cn}

\address{
	Lixin Yan, Department of Mathematics, Sun Yat-sen  University, Guangzhou, 510275, P.R. China}
\email{mcsylx@mail.sysu.edu.cn}

\date{\today}


	\begin{abstract}
		Let ${\frak M}^\alpha$ be    the spherical  maximal  operators of  complex  order $\alpha$ on
		${\mathbb R^n}$.
		In this article we show that when $n\geq 2$,  suppose
		\begin{eqnarray*}
			\|	{\frak M}^{\alpha}  f \|_{L^p({\mathbb R^n})} \leq C\|f \|_{L^p({\mathbb R^n})}
		\end{eqnarray*}
	holds for some $\alpha$ and $p\geq 2$,  then  we must  have
	$
		{\rm Re}\,\alpha \geq   \max \{1/p-(n-1)/2,\ -(n-1)/p \}.
$
When $n=2$, we prove that
  $	\|	{\frak M}^{\alpha}  f \|_{L^p({\mathbb R^2})} \leq C\|f \|_{L^p({\mathbb R^2})}$   if  ${\rm Re}\  \! \alpha>\max\{1/p-1/2,\ -1/p\}$, and hence
	the range of $\alpha$ is sharp  in the sense that  the estimate fails   for  ${\rm Re}\ \alpha <\max\{1/p-1/2, -1/ p\}.$
	
	\end{abstract}

	\maketitle

\section{Introduction}
\setcounter{equation}{0}

In 1976    Stein \cite{St}  introduced   the spherical  maximal  means ${\frak M}^\alpha f(x)= \sup_{t>0}
|{\frak M}^\alpha_tf (x) |$  of (complex) order $\alpha$,  where
\begin{eqnarray}\label{e1.1}
	{\frak M}^\alpha_t f (x) =
	{1\over   \Gamma(\alpha)  }  \int_{|y|\leq 1} \left(1-{|y|^2 }\right)^{\alpha -1} f(x-ty)\,\text{d}y.
\end{eqnarray}
These means are priori only defined for ${\rm Re}\  \alpha>0$, but the definition can be extended to   all complex $\alpha$  by analytic continuation.
In the  case $\alpha=1$,  ${\frak M}^\alpha$ corresponds to the Hardy-Littlewood maximal operator and in the case  $\alpha=0$, one  recovers the spherical maximal means  ${\frak M}  f(x)= \sup_{t>0}
|{\frak M}_tf (x) |$ in which
\begin{eqnarray}\label{e1.2}
{\frak M}_tf (x) =
{1\over   \Gamma(\alpha)  }  \int_{{\mathbb S}^{n-1}}   f(x-ty) \,\text{d}\sigma(y), \ \ \ (x, t)\in {\mathbb R^n}\times {\mathbb R^+},
\end{eqnarray}
where $\text{d}\sigma$ is the induced Lebesgue measure on the unit sphere ${\mathbb S}^{n-1}\subseteq {\mathbb R^n}.$ In \cite[Theorem 2]{St},
 Stein showed that
 \begin{eqnarray}\label{e1.3}
 	\|	{\frak M}^{\alpha}  f \|_{L^p({\mathbb R^n})} \leq C\|f \|_{L^p({\mathbb R^n})}
 \end{eqnarray}
in the following circumstances:
 \begin{eqnarray}\label{e1.4}
{\rm Re}\, \alpha>1-n+ {n\over p} \ \ \ \     {\rm when}\     1<p\leq 2;
\end{eqnarray}
or
 \begin{eqnarray}\label{e1.5}
\ \ {\rm Re}\, \alpha>{2-n\over p}   \ \ \ \   \ \ \ \  {\rm when}\   2\leq p\leq \infty.
\end{eqnarray}
The above maximal theorem tells us  that    when $n\geq 3$,  the  maximal operator ${\frak M} $ is bounded on $L^p({\mathbb R^n})$ for the   range of $p>n/(n-1)$. This range of $p$ is sharp,
as has been pointed out in \cite{St, SWa},  no such result can hold for $p\leq n/(n-1)$
if $n\geq 2$.

Some 10 years passed before  Bourgain \cite{Bo} finally proved that  the maximal operator    ${\frak M} $ is bounded on $L^p{(\mathbb R^2)} $ for $p>2$.    Bourgain's theorem says that there exists $\epsilon(p)>0$ such that
\begin{eqnarray}\label{e1.6}
\|	{\frak M}^{\alpha}  f \|_{L^p({\mathbb R^2})} \leq C\|f \|_{L^p({\mathbb R^2})}, \ \ \  {\rm Re}\, \alpha > -\epsilon(p), \ \ \  2< p< \infty.
\end{eqnarray}
This result cannot hold even for $\alpha=0$ when $p=2$, see \cite{St}.
An alternative proof of Bourgain's result was subsequently found by
  Mockenhaupt, Seeger and Sogge \cite{MSS}, who used   a  local smoothing estimate  for the solutions  of the   wave operator. In 2017,   Miao, Yang and Zheng \cite{MYZ} improved certain range of $\alpha$ for $L^p$-bounds for the operator 	${\frak M}^{\alpha}$   by using  the Bourgain-Demeter decoupling theorem  \cite{BD}.  All these refinements can be stated altogether as follows:
  For $  n\geq 2$ and $p\geq 2$,  \eqref{e1.3} holds
whenever
\begin{eqnarray}\label{e1.7}
{\rm Re} \, \alpha >\max\left\{{ 1-n\over 4} +{3-n\over 2p}, \,  {1-n\over p} \right\}.
\end{eqnarray}
The above range $\alpha$ in \eqref{e1.7} for $p>2$ is strictly wider than the range of $\alpha$ in \eqref{e1.5}. However, the range $\alpha$ in  \eqref{e1.7} is not optimal.

As mentioned above, the proof of the range  of $\alpha$ in  \eqref{e1.7}   relies on the   progress concerning Sogge's local smoothing conjecture, as originally formulated by Sogge \cite{S2}:
For $n\geq 2$ and $p\geq 2n/(n-1)$, one has
\begin{eqnarray}\label{e1.8}
	\left( \iint_{{\mathbb R^n}\times [1,2]} \big| u(x,t)\big|^p \,\text{d}x\text{d}t\right)^{1/p}  \leq C\|f\|_{W^{\gamma, p}({\mathbb R^n})}, \ \
	\ \ \
	{\rm if}\ \  \gamma>{n-1\over 2}-{n\over p},
\end{eqnarray}
where $u(x,t)$  is the solution of the Cauchy problem for the wave equation in ${\mathbb R^{n}\times {\mathbb R}}:$
\begin{eqnarray}\label{e1.9}
	\left \{
	\begin{array}{rll}
		\left((\partial/\partial t)^2-\Delta\right) u(x,t)  & =&0, \\[5pt]
		u|_{t=0}& =&0, \\[5pt]
		(\partial/\partial t) u|_{t=0}& =&f.
	\end{array}
	\right.	
\end{eqnarray}
The local smoothing conjecture has been studied in numerous papers, see for instance \cite{BD, GS,  GWZ, LW, MYZ, MSS,   S1,    W} and the references therein. When $n=2$, sharp results follow by the work of Guth, Wang and Zhang \cite{GWZ}. When $n\geq 3$, the conjecture holds for all $p\geq {2(n+1)/ (n-1)}$ by the Bourgain-Demeter decoupling theorem \cite{BD} and the method of \cite{W}.

The aim of this article is to prove the following result.

\begin{theorem}\label{th1.1} Let $p\geq 2$.
	\begin{enumerate}
		\item[\rm(i)] Let  $n\geq 2$.   Suppose \eqref{e1.3} holds  for some $\alpha\in\mathbb{C}$. Then  we must have
		$$
		{\rm Re}\,\alpha\geq  \max\left\{{1\over p}-{n-1\over 2},\ -{n-1\over p}\right\}.
		$$
	
		\medskip

		\item[\rm(ii)] Let  $n=2$.  Then  the estimate \eqref{e1.3} holds  if
		$$
		{\rm Re}\,\alpha>  \max\left\{{1\over p}-{1\over 2},\ -{1\over p}\right\}.
		$$
	\end{enumerate}
	
\end{theorem}

 Let $p\geq 2$ and  $\alpha =(3-n)/2$. For an appropriate constant $c_n$, we have that $c_nt ({\frak M}_t^\alpha f)(x)= u(x,t)$, where $u$ is the solution of the wave equation  \eqref{e1.9},  see \cite[4.10, p.519]{St1}.
As a consequence of (i) of Theorem~\ref{th1.1}, we have the following corollary.

\begin{corollary} \label{coro1.2}     Let $n\geq 4$.  Then
	$$
 \left \| \sup_{t>0} \left|{u(x,t)\over t}\right| \right\|_{L^p({\mathbb R^n})} \leq C_p\|f\|_{L^p({\mathbb R^n})}
$$
can not hold   whenever  $p> 2(n-1)/(n-3)$.	
\end{corollary}

We would like to mention that for the range $\alpha$ in \eqref{e1.5},   it is commented in \cite[4.10, p.519]{St1}   that the optimal results for $p>2$ and $n\geq 2$ ``are still a mystery".
Our   Theorem~\ref{th1.1}   gives an affirmative answer  in dimension $n=2$  to
show sharpness of  $	{\rm Re}\,\alpha > \max\left\{{1/p}-{1/2},\ -{1/p}\right\}$  in the estimate \eqref{e1.3} except the borderline.

The proof of (ii) of Theorem~\ref{th1.1} can be shown   by applying    the work  of Guth-Wang-Zhang \cite{GWZ} on  local smoothing estimates   along with the techniques previously used
in \cite{MSS} and \cite{MYZ}.
The main contribution of this article is to show  (i) of Theorem~\ref{th1.1}.
From the asymptotic expansion of Fourier multiplier of the operator ${\frak M}^{\alpha}_t$,
 it is seen  that ${\frak M}_t^\alpha $ are essentially the sum of half-wave operators
$e^{it\sqrt{-\Delta}} $ and $e^{-it\sqrt{-\Delta}}$, and hence the complexity  of the operator ${\frak M}^{\alpha}_t$
 comes from  the interference between the  operators $e^{it\sqrt{-\Delta}} $ and $e^{-it\sqrt{-\Delta}}$. To
show  the necessity of $L^p$-boundedness of   ${\frak M}^{\alpha}_t$,  we make  the following   observations.     For the case  $p>2n/(n-1)$ we note  that by the stationary phase argument, two waves  $e^{ it\sqrt{-\Delta}}f$  and $e^{ -it\sqrt{-\Delta}}f$  concentrate  on the opposite parts of sphere $\{x\in\Rn:  |x|=t\}$, respectively, when $\hat f$ is supported on a small cone.
 For the case  $2\leq p\leq 2n/(n-1)$, we let $f$   be a wave packet of direction $\nu\in S^{n-1}$, then one can regard $e^{\pm it\sqrt{-\Delta}}f(x)$ as the translations $f(x\pm t\nu)$ of $f(x)$, which  concentrate  on the opposite parts of sphere $\{x\in\Rn:\ |x|=t\}$.    In Section 3, we
construct two examples  such that there is no interference between $e^{it\sqrt{-\Delta}}f$ and $e^{-it\sqrt{-\Delta}}f$ to   obtain
  the desired range of $\alpha$   in
(i) of Theorem~\ref{th1.1}.

The paper is organized as follows. In Section 2, we give some preliminary results including the properties of
the Fourier multiplier associated to the    spherical   operators $	{\frak M}^\alpha_t $ by using asymptotic expansions of Bessel functions.
The proof of (i) of Theorem~\ref{th1.1} will be given in  Section 3 by constructing two examples to obtain the necessarity of $L^p$-bounds for the maximal operator $	{\frak M}^\alpha$.
In Section 4 we will give the proof of (ii) of Theorem~\ref{th1.1}.
  In the Appendix, we clarify certain estimates for the Bessel function
   used in Section 2
for the  convenience of reader.


\medskip

\section{Preliminary results}\label{s2}
\setcounter{equation}{0}

We begin with recalling the  spherical   function  $	{\frak M}^\alpha_t f(x)=f\ast m_{\alpha,t}(x)$
where $m_{\alpha,t}(x)= t^{-n}m_{\alpha}(t^{-1}x)$
and
$$
	m_{\alpha}(x)=\Gamma(\alpha)^{-1}\big(1-|x|^{2}\big)_{+}^{\alpha-1},
$$
where $\Gamma(\alpha)$ is the Gamma function and $(r)_+=\max\{0,r\}$ for $r\in \mathbb R$. Define the Fourier transform of $f$ by
$ 
\hat{f}(\xi)=\int_{\Rn} e^{2\pi ix \cdot\xi}f(x)\,\text{d}x.
$ 
It follows by    \cite[p.171]{SW}  that the Fourier transform of $m_\alpha$ is given by
\begin{align}\label{e2.1}
	\widehat{m_{\alpha}}(\xi)=\pi^{-\alpha+1}|\xi|^{-n/2-\alpha+1}
	J_{n/2+\alpha-1}\big(2\pi|\xi|\big).
\end{align}
Here $J_\beta$ denotes the Bessel function of order $\beta$.
For ${\rm Re}\, \beta>-1/2$, we can obtain the complete
 asymptotic expansion
\begin{align}\label{ebb}
	J_\beta(r)\sim  r^{-1/2}e^{ir} \sum_{j= 0}^{\infty}b_{j} r^{-j} + r^{-1/2}e^{-ir} \sum_{j= 0}^{\infty}d_j r^{-j},\, \ \ \ \ \ \  r\geq 1
\end{align}
for suitable coefficients $b_j$ and $d_j$.
Note that when  $\beta$ is a positive integer,  \eqref{ebb} is given in \cite[(15), p.338]{St1}. For ${\rm Re}\, \beta>-1/2$, we refer it to the Appendix.
 For any given $N\geq 1$, it follows from \eqref{ebb} that
\begin{align}\label{e2.2}
	J_\beta(r)= r^{-1/2}e^{ir}\left(\sum_{j= 0}^{N-1}b_{j} r^{-j}+E_{N,1}(r)\right)+ r^{-1/2}e^{-ir}\left(\sum_{j= 0}^{N-1}d_j r^{-j}+ E_{N,2}(r)\right)+E(r)
\end{align}
for $r\geq 1$. Here  $b_0,d_0\neq 0$,  and the error terms satisfy

\begin{align}\label{e2.3}
	\bigg|\bigg(\frac{d}{dr}\bigg)^k E_{N,1}(r)\bigg|+\bigg|\bigg(\frac{d}{dr}\bigg)^k E_{N,2}(r)\bigg|+\bigg|\bigg(\frac{d}{dr}\bigg)^k E(r)\bigg|\leq C_{k}r^{-N-k},
\end{align}
for all $k\in \Z_+$.
We rewrite  \eqref{e2.1} as
\begin{align}\label{e2.4}
	\widehat{m_\alpha}(\xi)
	&= \varphi(|\xi|)\widehat{m_\alpha}(\xi)
	 + (1-\varphi(|\xi|))\widehat{m_\alpha}(\xi)   \notag\\
	&= \left[ \varphi(|\xi|)\widehat{m_\alpha}(\xi)+{\mathcal E}(|\xi|) \right ] \notag\\
	&+ \left[e^{2\pi i|\xi|}{\mathcal E}_{N,1}(|\xi|)+e^{-2\pi i|\xi|}{\mathcal E}_{N,2}(|\xi|) \right]\notag\notag\\
	& + |\xi|^{-(n-1)/2-\alpha}\left[e^{2\pi i|\xi|}a_1(|\xi|)+e^{-2\pi i|\xi|}a_2(|\xi|) \right],
\end{align}
where
\begin{align}\label{ee2.1}
&{\mathcal E}(r) = (2\pi)^{1/2}c(\pi, \alpha)(1-\varphi(r))r^{-(n-2)/2-\alpha}E(2\pi r),\nonumber\\
&{\mathcal E}_{N,\ell}(r) = c(\pi, \alpha) E_{N,\ell}(2\pi r)(1-\varphi(r))r^{-(n-1)/2-\alpha},\ \ \ \ell=1,2, \nonumber\\
&a_1(r) =  c(\pi, \alpha) \sum_{j= 0}^{N-1}b_j (2\pi r)^{-j}(1- \varphi(r)),\\
 &a_2(r) =   c(\pi, \alpha)\sum_{j= 0}^{N-1}d_j (2\pi r)^{-j}(1- \varphi(r))\nonumber
\end{align}
with $c(\pi, \alpha)= 2^{-1/2}\pi^{-\alpha+1/2}$.
Here  $\varphi\in C_0^{\infty}(\R)$ is   an even function, identically equals $1$ on
$B(0,M)$  and supported on $B(0,2M)$, where $M=M(N)$ is large enough
such that $|a_2(r)| \geq c_{low}>0 $ for $|r|\geq M$.
 Then we can split the Fourier multiplier of the operator $\frak M^\alpha_1 $ into three parts as in \eqref{e2.4} above.
 Firstly, we note that
  $\varphi(|\xi|)\widehat{m_\alpha}(\xi)$ is smooth and compactly supported and ${\mathcal E}(|\xi|)\in {\mathscr S}(\Rn)$. It is seen that  $\sup_{t>0} |\widehat{m_\alpha}(tD) \varphi(t|D|)f|$ and $\sup_{t>0} |{\mathcal E}(t|D|)f|$ are bounded by the Hardy-Littlewood maximal function. Hence,
	\begin{eqnarray}\label{e2.5} \hspace{1cm}
		\big\|\sup_{t>0} |\widehat{m_\alpha}(tD) \varphi(t|D|))f|\big \|_{L^p(\R^n)}+ \big\|\sup_{t>0} |{\mathcal E}(t|D|)f |\big \|_{L^p(\R^n)} \leq C\|f\|_{L^p(\Rn)}, \ \ \ \  p>1.
	\end{eqnarray}
Secondly,  we define
$$
{\mathscr E}_{N}f(x,t)=\int_{\Rn} e^{2\pi i(x\cdot \xi+t|\xi|)}{\mathcal E}_{N,1}\big(t|\xi|\big) \hat{f}(\xi)\,\text{d}\xi+\int_{\Rn} e^{2\pi i(x\cdot \xi-t|\xi|)}{\mathcal E}_{N,2}\big(t|\xi|\big)  \hat{f}(\xi)\,\text{d}\xi.
$$
Then   we have the following lemma.
\begin{lemma}\label{le2.1}
	Let $p\ge 2$.  There exists a constant $C>0$ such that
	\begin{eqnarray}\label{e2.6}
		\big\|\sup_{t\in [1,2]}|{\mathscr E}_{N}f(\cdot,t)| \big\|_{L^p(\Rn)}\leq C  \|f\|_{L^p(\Rn)},
	\end{eqnarray}
	when
	$$
	N>-{n-2\over p}-{\rm Re}\,\alpha.
	$$
\end{lemma}

The proof of Lemma~\ref{le2.1} is based on the following elementary result (see \cite[Lemma 2.4.2]{S1}).

\begin{lemma}\label{le2.2}
Let $F$ be a smooth function defined on  ${\mathbb R^n}\times [1,2]$.  Then  for  $p>1$ and $1/p +1/p' =1,$
	$$
	\big\|\sup_{1\leq t\leq 2} |F(\cdot,t)|\big\|_{L^p(\Rn)}\leq C_p\left( \|F(\cdot,1)\|_{L^p(\Rn)} + \|F\|_{L^p(\Rn\times  [1,2])}^{1-1/p} \left\|\partial_t F\right\|_{L^p(\Rn\times  [1,2])}^{1/p}\right).
	$$
\end{lemma}

\smallskip

\begin{proof}[Proof of Lemma~\ref{le2.1}]
We fix a function $\varphi$ as in \eqref{e2.4}. For $j\geq 1$, we let
 $ \psi_j(r)=\varphi(2^{-j} r)-\varphi(2^{-j+1} r)$  such  that $
  1\equiv  \varphi(r) +  \sum_{j\geq 1 }\psi_j(r), r\geq0.
  $
	For $j\geq1$, define
	$$
	{\mathscr E}_{N,j}f(x,t)=\int_{\mathbb R^n}  \left(e^{2\pi i(x\cdot \xi+t|\xi|)}{\mathcal E}_{N,1}\big(t|\xi|\big) + e^{2\pi i(x\cdot \xi-t|\xi|)}{\mathcal E}_{N,2}\big(t|\xi|\big)\right)  {\psi_j}(t|\xi|)\hat{f}(\xi)d\xi.
	$$
	To prove  \eqref{e2.6}, it suffices to show that there exists a constant  $\delta>0$ such that
	for all $j\geq 1$,
	\begin{eqnarray}\label{e2.7}
		\big\|\sup_{1\leq t\leq 2}|{\mathscr E}_{N,j}f(\cdot,t)| \big\|_{L^p(\Rn)}\leq C 2^{-\delta j}\|f\|_{L^p(\Rn)}.
	\end{eqnarray}

	Let us prove \eqref{e2.7} by using Lemma \ref{le2.2}.  First, for each fixed $t\in[1,2]$, ${\mathscr E}_{N,j}f$ are the sum of two Fourier integral operators of order $-(n-1)/2-{\rm Re}\, \alpha-N$ with phase $x\cdot\xi\pm t|\xi|$.
	By \cite[Theorem 2, Chapter IX]{St1}  and the fact that $e^{it\sqrt{-\Delta}}$ is local at scale $t$, we have
	\begin{align}\label{e2.8}
	  \sup_{1\leq t\leq2}\big\|	 {\mathscr E}_{N,j}f(\cdot,t)\big\|_{L^p(\Rn)}\leq C 2^{-((n-1)/2+{\rm Re}\,\alpha+N) j}2^{(n-1)(1/2-1/p)j}\|f\|_{L^p(\Rn)},
	\end{align}
	see also \cite[Corollary 2.4]{SSS}. Next, we write 	 $\partial_{t}{\mathscr E}_{N,j}f(x,t)$ as the sum of following terms,
	\begin{eqnarray*}
		&&	\pm2\pi it^{-1}\int e^{2\pi i(x\cdot \xi\pm t|\xi|)}t|\xi|{\mathcal E}_{N,1}\big(t|\xi|\big) {\psi_j}(t|\xi|)\hat{f}(\xi)d\xi;\\
		&&
			 	\pm2\pi it^{-1}\int e^{2\pi i(x\cdot \xi\pm t|\xi|)}t|\xi|{\mathcal E}_{N,2}\big(t|\xi|\big) {\psi_j}(t|\xi|)\hat{f}(\xi)d\xi;\\
		&&
		t^{-1}\int e^{2\pi i(x\cdot \xi\pm t|\xi|)}t|\xi|({\mathcal E}_{N, 1}\psi_{j})'\big(t|\xi|\big) \hat{f}(\xi)d\xi;\\
		&&
		t^{-1}\int e^{2\pi i(x\cdot \xi\pm t|\xi|)}t|\xi|({\mathcal E}_{N, 2}\psi_{j})'\big(t|\xi|\big) \hat{f}(\xi)d\xi.
	\end{eqnarray*}
	By \eqref{e2.3}, we see that for each fixed $t\in[1,2]$, they are Fourier integal operators of order no more than $-(n-1)/2-{\rm Re}\, \alpha-N+1$.
	By \cite[Theorem 2, Chapter IX]{St1} again,
	\begin{align}\label{e2.9}
		\sup_{1\leq t\leq2}\big\|\partial_t{\mathscr E}_{N,j}f(\cdot,t) \big\|_{L^p(\Rn)}\leq C 2^{-((n-1)/2+{\rm Re}\,\alpha+N-1) j}2^{(n-1)(1/2-1/p)j}\|f\|_{L^p(\Rn)}.
	\end{align}
Lemma \ref{le2.2}, together with \eqref{e2.8} and  \eqref{e2.9},  gives
	$$
	\big\|\sup_{1\leq t\leq 2}|{\mathscr E}_{N,j}f(\cdot,t)| \big\|_{L^p(\Rn)}\leq C 2^{-((n-1)/2+{\rm Re}\,\alpha+N-1/p) j}2^{(n-1)(1/2-1/p)j}\|f\|_{L^p(\Rn)}.
	$$
	Choosing $N>-(n-2)/p-{\rm Re}\,\alpha$ and letting    $\delta= N+(n-2)/p+{\rm Re}\,\alpha$, we  obain  estimate \eqref{e2.7}. The  proof of Lemma~\ref{le2.1} is complete.
\end{proof}

Finally, we define
\begin{align}\label{e2.10}
	{\mathscr A}_t f(x)= \int_{\Rn} \left( e^{2\pi i(x\cdot\xi+ t|\xi|)}a_1(t|\xi|) +  e^{2\pi i(x\cdot\xi- t|\xi|)}a_2(t|\xi|) \right) \hat{f}(\xi) \,\text{d}\xi.
\end{align}
From \eqref{e2.4}, \eqref{e2.5} and Lemma~\ref{le2.2}, we see that the $L^p$-boundness of
the operator $\frak M^\alpha_t $  reduces to  boundedness of the operator
$	{\mathscr A}_t $ on Sobolev spaces,
which will be investigated  in Section 3 below.

\medskip

\section{Proof of  {\rm (i)} of   Theorem~\ref{th1.1}}
\setcounter{equation}{0}


To prove  {\rm (i)} of   Theorem~\ref{th1.1}, we need to show the following proposition.

\begin{proposition}\label{prop3.1}
	Let $n\geq 2$ and $p>2n/(n-1)$. Suppose
	\begin{align}\label{e3.1}
		\big\| {\frak M}_1^\alpha f\big\|_{L^p(\Rn)}\leq C\|f\|_{L^p(\Rn)}
	\end{align}
	holds for some $\alpha\in \mathbb{C}$. Then, we have
	$${\rm Re}\,\alpha\geq -{n-1\over p}.
	$$
\end{proposition}

 Let us prove Proposition~\ref{prop3.1}. By interpolation with ${\rm Re}\,\alpha=1$, it suffices to show that \eqref{e3.1} can not hold if ${\rm Re}\,\alpha\in (-(n-1)/2, -(n-1)/p )$. Here ${\rm Re}\,\alpha>-{(n-1)/2}$ is to make sure that we can use \eqref{e2.2}.
By \eqref{e2.5} and Lemma~\ref{le2.1}, if $N$ in Lemma~\ref{le2.1}  is chosen large enough,  then  the proof of Proposition~\ref{prop3.1} reduces to  the following lemma.

\begin{lemma}\label{le3.2}
	Let $n\geq 2$ and $1<p<\infty$. Suppose
	\begin{align}\label{e3.2}
		\|{\mathscr A}_1f \|_{L^p(\Rn)}\leq C\|f\|_{W^{s,p}(\Rn)}
	\end{align}
	holds for some $s\in\R$. Then, we have $$s\geq (n-1)\left|\frac12-\frac1p\right|.$$
\end{lemma}

\begin{proof}
	Let $\widehat{\gamma_\beta}(\xi) :=(1+|\xi|^{2})^{-\beta/2}$ 	and  $\varphi$ as in \eqref{e2.4}.     Let $w $ belong to ${\it S}^{0}$ (a symbol of order zero)    satisfying $ |w(r)|\geq c>0   $ on $\R$
	for some constant $c$. 
Moreover, $w $ equals $\big(\sum_{j\geq 0}^{N-1}d_j r^{-j}\big)^{-1}$ on  $\supp (1-\varphi)$, and equals constant near zero. Assume that $\chi(\xi)\in C^\infty(\Rn\backslash 0)$ is homogeneous of order $0$ and vanishes if $|\xi/|\xi|-v_1|\geq 10^{-2}$, where $v_1:=(1,0,\cdots,0)$. Define
	$$ \hat{f}_\beta(\xi)=w(|\xi|)\chi(\xi)\widehat{\gamma}_\beta(\xi).
	$$
	Since $w(|\xi|)\in S^{0}$ and $\chi$ is a H\"ormander multiplier, $w(|D|)$ and $\chi(D)$ are bounded on $L^{p}(\Rn)$, we have
	\begin{align}\label{e3.3} \|f_\beta\|_{W^{s,p}(\Rn)}=\|w(|D|)\chi(D)\gamma_{\beta-s}\|_{L^{p}(\Rn)}\leq C\|\gamma_{\beta-s}\|_{L^p(\Rn)}.
	\end{align}
	On the other hand, 	it follows  by \cite[Proposition 1.2.5]{G} that
	\begin{eqnarray*}
		|\gamma_{\beta-s}(x)|\leq \left\{
		\begin{array}{lll}
			C|x|^{-n+\beta-s} & {\rm if} & \, |x|\leq2,\\[4pt]
			Ce^{-|x|/2} & {\rm if} & \, |x|\geq2
		\end{array}
		\right.
	\end{eqnarray*}
	when  $0<\beta-s<n$.
	This implies that  whenever $0<\beta-s<n$ and  $(-n+\beta-s)p>-n$,  \eqref{e3.3} is finite.

	By the expansion in \cite[p. 360]{St1}, we can write for $|x|\geq 1$ and $|x/|x|+v_1|\leq 10^{-2}$ that
	\begin{align}\label{e3.4}
		 \widehat{\chi\text{d}\sigma}(x)= e^{2\pi i|x|}h(x)+e(x),
	\end{align}
	where $e$ belongs to ${\it S}^{-\infty}$ and $h\in {\it S}^{-(n-1)/2}$ can be splited into two terms:
	\begin{eqnarray}\label{e3.5}
		h(x)= c_0|x|^{-(n-1)/2}\chi(-x/|x|)+ \tilde{e}(x),\ \ \tilde{e}\in {\it S}^{-(n+1)/2}
	\end{eqnarray}
for all $|x|\geq1$.
	Using polar coordinate and \eqref{e3.4}, we see that if $|x/|x|-v_1|\leq 10^{-2}$,
	\begin{align}\label{e3.6}
		\int_{\Rn}  e^{2\pi i(x\cdot\xi+ |\xi|)}a_1(\xi)    \hat{f_\beta}(\xi) \,\text{d}\xi
		 &=\int_{0}^{\infty}\int_{S^{n-1}} e^{2\pi i(x\cdot r\theta+ r)}a_{1}(r)w(r) \chi(\theta)(1+r^{2})^{-\beta/2}r^{n-1}\,\text{d}r\text{d}\sigma(\theta)\nonumber\\
		&=\int_{0}^{\infty} e^{2\pi ir}\widehat{\chi\text{d}\sigma}(-rx)a_1(r)w(r) (1+r^{2})^{-\beta/2}r^{n-1}\,\text{d}r\nonumber\\
		&=\int_{0}^{\infty} e^{2\pi ir(|x|+1)}h(-rx)a_1(r)w(r) (1+r^{2})^{-\beta/2}r^{n-1}\,\text{d}r\nonumber \\
		& +\int_{0}^{\infty} e^{2\pi ir}e(-rx)a_1(r)w(r) (1+r^{2})^{-\beta/2}r^{n-1}\,\text{d}r.
	\end{align}
	By integration by parts, we have that 	if $|x/|x|-v_1|\leq 10^{-2}$ and $ 1/2\leq |x|\leq 2$,
	\begin{align}\label{e3.7}
		\left|	\int_{\Rn}  e^{2\pi i(x\cdot\xi+ |\xi|)}a_1(\xi)    \hat{f_\beta}(\xi) \,\text{d}\xi\right|\leq C.
	\end{align}

	Next  we calculate
	$$
	\int_{\Rn} e^{2\pi i(x\cdot\xi- |\xi|)}a_2(\xi) \hat{f_\beta}(\xi) \,\text{d}\xi
	$$ when $|x/|x|-v_1|\leq 10^{-2}$ and $||x|-1|\leq\varepsilon$ ( $\varepsilon>0$ is a small constant that will be chosen later). 	As \eqref{e3.6},  we write
	\begin{align*}
		\int_{\Rn} e^{2\pi i(x\cdot\xi- |\xi|)}a_2(\xi) \hat{f_\beta}(\xi) \,\text{d}\xi
		 &=C\int_{0}^{\infty}\int_{S^{n-1}} e^{2\pi i(x\cdot r\theta- r)}(1-\varphi(r)) \chi(\theta)(1+r^{2})^{-\beta/2}r^{n-1}\,\text{d}r\,\text{d}\sigma(\theta)\\
		 &=C\int_{0}^{\infty} e^{-2\pi ir}\widehat{\chi\text{d}\sigma}(-rx)(1-\varphi (r)) (1+r^{2})^{-\beta/2}r^{n-1}\,\text{d}r\\
		 &=C\int_{0}^{\infty} e^{2\pi ir(|x|-1)}h(-rx)(1-\varphi(r)) (1+r^{2})^{-\beta/2}r^{n-1}\,\text{d}r\\
		& +C\int_{0}^{\infty} e^{-2\pi ir}e(-rx)(1-\varphi(r)) (1+r^{2})^{-\beta/2}r^{n-1}\,\text{d}r.
	\end{align*}
	 The second term is bounded since $e\in S^{-\infty}$.
	Now we use  \eqref{e3.5} to write
	\begin{align*}
	&	\int_{\Rn} e^{2\pi i(x\cdot\xi- |\xi|)}a_2(\xi) \hat{f_\beta}(\xi) \,\text{d}\xi\\
	\hspace{0.8cm} 	 &=C\int_{0}^{\infty}  e^{2\pi ir(|x|-1)} \big[ c_0(r|x|)^{-\frac{n-1}{2}}\chi(x/|x|)   +  \tilde{e}(-rx)  \big](1-\varphi (r)) (1+r^{2})^{-\beta/2}r^{n-1}\,\text{d}r+O(1).
	\end{align*}
	
	To continue, we need the following result.
	
	\begin{lemma}\label{le3.3} Suppose  $|g^{(k)}(r)|\leq Cr^{m-k}$ for all $k\in\Z_+$. Then we have
		\begin{align}\label{e3.8}
			\bigg|\int_0^\infty e^{2\pi ir\tau} g(r)(1-\varphi(r)) \,\text{d}r\bigg|\leq C|\tau|^{-m-1}.
		\end{align} 		
	\end{lemma}
	
	\begin{proof}
		We write
		$$
		\int_0^\infty e^{2\pi ir\tau} g(r)(1-\varphi(r)) \,\text{d}r= \sum_{j\geq 1}\int_0^\infty e^{2\pi ir\tau} g(r)\psi(r/2^j) \,\text{d}r.
		$$
		For each $j$ and $N$, integration by parts shows
		\begin{align}\label{e3.9}
			\bigg|  \int_0^\infty e^{2\pi ir\tau} g(r)\psi(r/2^j) \,\text{d}r  \bigg|
			&= (2\pi)^{N}|\tau|^{-N}\bigg|\int_0^\infty e^{2\pi ir\tau}\bigg( \frac{d}{dr} \bigg)^N\big( g(r)\psi(r/2^j)\big) \text{d}r \bigg|       \notag \\
			&\leq C|\tau|^{-N}\int_{2^{j-1}\leq r\le 2^{j+1}} r^{m-N} \text{d}r \notag \\
			&\leq C|\tau|^{-N} 2^{j(m-N+1)}.
		\end{align}
 Setting $N=0$ for $2^j\le |\tau|^{-1}$, and  $N>m+1$
 otherwise. From this, it follows that
		\begin{align*}
			\bigg|\int_0^\infty e^{2\pi ir\tau} g(r)(1-\varphi(r)) \,\text{d}r\bigg|
			&\leq C\sum_{2^j\le |\tau|^{-1}} 2^{j(m+1)}+ C\sum_{2^j\ge |\tau|^{-1}}|\tau|^{-N} 2^{j(m-N+1)}\\
			&\leq C|\tau|^{-m-1}.
		\end{align*}
		This proves estimate  \eqref{e3.8}.
	\end{proof}
	
	\medskip
	
	\noindent
	{\bf Back to the proof of Lemma~\ref{le3.2}}. By Lemma~\ref{le3.3},
	$$
		\int_{0}^{\infty}  e^{2\pi ir(|x|-1)}   \tilde{e}(-rx)  (1-\varphi (r)) (1+r^{2})^{-\beta/2}r^{n-1}\,\text{d}r
	 =O\left(\big||x|-1\big|^{\beta-(n-1)/2}\right).
	$$
	Finally,  for $|x/|x|-v_1|\leq 10^{-2}$ and $||x|-1|\leq\varepsilon$, let us  estimate
	$$
	\int_{0}^{\infty} e^{2\pi ir(|x|-1)}(r|x|)^{-\frac{n-1}{2}}(1-\varphi (r)) (1+r^{2})^{-\beta/2}r^{n-1}\,\text{d}r.
	$$
	Note that by Lemma~\ref{le3.3} again, 	
	$$
	 |x|^{-\frac{n-1}{2}}\int_{0}^{\infty} e^{2\pi ir(|x|-1)}(1-\varphi(r)) r^{\frac{n-1}{2}}((1+r^{2})^{-\beta/2}-r^{-\beta})\,\text{d}r
	 =O\bigg(\big||x|-1\big|^{\beta-(n-1)/2+1}\bigg).
	$$
	For the term $	 						
	 |x|^{-\frac{n-1}{2}}\int_{0}^{\infty} e^{2\pi ir(|x|-1)}(1-\varphi(r)) r^{-\beta+\frac{n-1}{2}}\,\text{d}r,
	$
	we use  scaling to obtain that if $-\beta+\frac{n-1}{2}>-1$,
	\begin{align*}
		 |x|^{-\frac{n-1}{2}}\int_{0}^{\infty} e^{2\pi ir(|x|-1)}(1-\varphi(r)) r^{-\beta+\frac{n-1}{2}}\,\text{d}r&=|x|^{-\frac{n-1}{2}}\int_{0}^{\infty} e^{2\pi ir(|x|-1)} r^{-\beta+\frac{n-1}{2}}\,\text{d}r+O(1)\\
&=|x|^{-\frac{n-1}{2}}(|x|-1)^{\beta-\frac{n+1}{2}}\int_{0}^{\infty} e^{2\pi ir} r^{-\beta+\frac{n-1}{2}}\,\text{d}r+O(1).
	\end{align*}
Hence, there exists $\varepsilon_{1}\in(0,1/2)$ such that if $||x|-1|\leq \varepsilon_{1}$,
	$$
	 |x|^{-\frac{n-1}{2}}\bigg|\int_{0}^{\infty} e^{2\pi ir(|x|-1)}(1-\varphi(r)) r^{-\beta+\frac{n-1}{2}}\,\text{d}r\bigg|\geq C \big||x|-1\big|^{\beta-\frac{n+1}{2}}.
	$$
	Furthermore, we can find $0<\varepsilon\leq \varepsilon_1$ such that for $||x|-1|\le \varepsilon$,
	\begin{align*}
		\left|\int_{\Rn} e^{2\pi i(x\cdot\xi- |\xi|)}a_2(\xi) \hat{f_\beta}(\xi) \,\text{d}\xi\right|
		&\geq C \big||x|-1\big|^{\beta-\frac{n+1}{2}}-O\bigg(\big||x|-1\big|^{\beta-(n-1)/2+1}\bigg)\\
		&\geq C\big||x|-1\big|^{\beta-\frac{n+1}{2}} /2.
	\end{align*}

	This, together with \eqref{e3.7}, tells us that $f_{\beta}\in W^{s,p}(\Rn)$ and ${\mathscr A}_1(f_{\beta})\notin L^{p}(\Rn)$ if
	\begin{align}\label{e3.10}
		\left\{
		\begin{array}{rcl}
			&0<\beta-s<n, \quad   \\[4pt]
			&(-n+\beta-s)p>-n, \quad   \\[4pt]
			&-\beta+\frac{n-1}{2}>-1, \quad   \\[4pt]
			 &\big(\beta-(n+1)/2\big)p\leq -1, \\
		\end{array}
		\right.
	\end{align}
	which is solvable when $s< (n-1)(1/p-1/2)$. Hence,  if \eqref{e3.2} holds, then  we must have $s\geq (n-1)(1/p-1/2)$. By duality,
	\begin{align*}
		\|({\mathscr A}_1)^{*} f\|_{L^{p'}(\Rn)}\leq C\|f\|_{W^{s,p'}(\Rn)}.
	\end{align*}
 Because $({\mathscr A}_1)^{*}$ is essentially the same as ${\mathscr A}_1$, we must have $s\geq (n-1)(1/p'-1/2)=(n-1)(1/2-1/p)$ by the previous counterexample.  This proves Lemma~\ref{le3.2}, and then the proof of  Proposition~\ref{prop3.1} is complete.
\end{proof}

Next, let us    prove the following result.
\begin{proposition}\label{prop3.4}
	Let $n\geq 2$ and $p\geq 2$. Suppose
	\begin{align}\label{e3.11}
		\big\| \sup_{1\leq t\leq 2}|\frak{M}^\alpha_t f| \big\|_{L^p(\Rn)}\leq C\|f\|_{L^p(\Rn)}
	\end{align}
	holds for some $\alpha\in \mathbb{C}$. Then, we have
	$$
	{\rm Re}\,\alpha\geq \frac1p-\frac{n-1}{2}.
	$$
\end{proposition}

Let us prove Proposition~\ref{prop3.4}.
By an  interpolation  with ${\rm Re}\,\alpha=1$  it suffices to show that \eqref{e3.11} can not hold if ${\rm Re}\,\alpha\in (-{n-1\over 2},\frac1p-\frac{n-1}{2})$. Here ${\rm Re}\,\alpha>-{n-1\over 2}$ is to make sure that we can use \eqref{e2.2}.
By \eqref{e2.5} and Lemma~\ref{le2.1},  the proof of  Proposition \ref{prop3.4}
reduces to show  the following lemma.

\begin{lemma}\label{le3.5}
	Let $n\geq 2$ and $p>1$. Suppose
	\begin{align}\label{e3.12}
		\big\|\sup_{1\leq t\leq 2}|{\mathscr A}_t f|\big\|_{L^p(\R^n)}\leq C\|f\|_{W^{s,p}(\Rn)}
	\end{align}
	holds for some $s\in\R$. 	Then,  we have $s\geq 1/p$.
\end{lemma}

\begin{proof}
	Let $\delta>0$ be a small number to be chosen later, and denote  $\xi= (\xi_1, \xi^\prime)\in \Rn$. For a given  large $j\in {\mathbb N}$, we let $\hat{f}\geq0$  be a smooth cut-off of
	the set
	\begin{align}\label{e3.13}
		 \left\{(\xi_1,\xi^\prime)\in\Rn:|\xi_1-2^j|\leq \delta 2^{j-1}, |\xi^\prime|\leq \delta 2^{j/2}\right\}
	\end{align}
	such that $\big|\partial_{\xi}^{\beta} \hat f(\xi)\big|\leq C_{\delta,\beta} 2^{-j|\beta'|/2}2^{-j|\beta_{1}|}$ for any $\beta=(\beta_{1},\beta')\in \mathbb{Z}_{+}^{n}$.
	By a simple calculation, we see that
	\begin{align}\label{e3.14}
		|\xi|-\xi_1\leq C\delta^2
	\end{align}
	in the support of $\hat{f}$. Let $j$ be large enough such that $(1-\varphi(t|\xi|))\hat{f}(\xi)=\hat{f}(\xi)$ for all $t\in[1,2]$, $\xi\in\Rn$ and
\begin{align}\label{e3.15-315}
\inf_{\xi\in\supp \hat f}|a_2(\xi)|\geq c_{low}>0.
\end{align}
Note by \cite[Chapter IX, Section 4]{St1} we have
	$$\sup_{1\leq t\leq2}\big|\partial_{\xi}^{\beta} \big(e^{2\pi it(|\xi|-\xi_1)}a_1(t|\xi|)\hat f(\xi)\big)\big|\leq C_{\delta,\beta} 2^{-j|\beta'|/2}2^{-j|\beta_{1}|}.$$
	Then for $1\leq t\leq 2$ and $x_1>0$, we use integration by parts to bound that
	\begin{align}\label{e3.15}
		\bigg|\int_{\Rn} e^{2\pi i(x\cdot\xi+ t|\xi|)}a_1(t|\xi|)  \hat{f}(\xi) \,\text{d}\xi\bigg|
		&= \bigg|\int_{\Rn}  e^{2\pi i(x+te_1)\cdot\xi}\bigg(e^{2\pi it(|\xi|-\xi_1)}a_1(t|\xi|) \hat{f}(\xi)\bigg) \,\text{d}\xi \bigg|     \\
		&\leq C_{\delta} 2^{-jN}2^{j\frac{n+1}{2}}(x_1+t)^{-N}\leq C_{\delta} 2^{-jN}2^{j\frac{n+1}{2}},            \notag
	\end{align}
	where $N\geq1$ and the constant $C_\delta$ is independent of $j$ and $t$.

	As for $\int_{\Rn} e^{2\pi i(x\cdot\xi- t|\xi|)}a_2(t|\xi|) \hat{f}(\xi) \,\text{d}\xi$ with $1\leq t\leq 2$, we split it into three terms
	\begin{align}\label{e3.16}
		\int_{\Rn} e^{2\pi i(x\cdot\xi- t|\xi|)}a_2(t|\xi|) \hat{f}(\xi) \,\text{d}\xi
		&= \int_{\Rn}  e^{2\pi i(x-te_1)\cdot\xi}\big(e^{2\pi it(-|\xi|+\xi_1)}-1\big) a_2(t|\xi|)\hat{f}(\xi) \,\text{d}\xi \nonumber \\
		&+ \int_{\Rn}  (e^{2\pi i(x-te_1)\cdot\xi}-1)  a_2(t|\xi|)\hat{f}(\xi)\,\text{d}\xi \nonumber \\
		&+\int_{\Rn}  a_2(t|\xi|)\hat{f}(\xi) \,\text{d}\xi.
	\end{align}
	By \eqref{e3.14}, the first term of \eqref{e3.16} is bounded by
	$$
	C\int_{\Rn}  \big|t(-|\xi|+\xi_1)\big| \hat{f}(\xi) \,\text{d}\xi\leq C\delta^{2}\int_{\Rn}  \hat{f}(\xi) \,\text{d}\xi\leq C\delta^{n+2} 2^{j(n+1)/2}.
	$$
	If $|x_1-t|\leq \delta 2^{-j}$ and $|x^\prime|\leq 2^{-j/2}$, by the support condition \eqref{e3.13} of $\hat{f}$, we have
	$$
	\big|(x-tv_1)\cdot\xi\big|\leq C\delta, \, \textrm{for all}\, \xi\in\supp \hat{f},
	$$
	which implies the second term of \eqref{e3.16} is bounded by
	\begin{align*}
		C\int_{\Rn}  \big|(x-tv_1)\cdot\xi\big|  \hat{f}(\xi)\,\text{d}\xi\leq C\delta\int_{\Rn}  \hat{f}(\xi)\,\text{d}\xi\leq C\delta^{n+1} 2^{j(n+1)/2}.
	\end{align*}
	By \eqref{e3.15-315}, we have
	\begin{align*}
		\bigg|\int_{\Rn}  a_2(t|\xi|)\hat{f}(\xi) \,\text{d}\xi\bigg|\geq \frac{c_{low}}{2}\int_{\Rn}  \hat{f}(\xi) \,\text{d}\xi\geq C_{L}\delta^{n}2^{j\frac{n+1}{2}}.
	\end{align*}
	Then by \eqref{e3.16} and the above estimates, if $\delta\leq \min\{\frac{C_{L}}{2C_{U}},1\}$, we have
	\begin{align}\label{e3.17}
		\bigg|\int_{\Rn} e^{2\pi i(x\cdot\xi- t|\xi|)}a_2(t|\xi|) \hat{f}(\xi) \,\text{d}\xi\bigg|\geq\bigg|\int_{\Rn}  a_2(t|\xi|)\hat{f}(\xi) \,\text{d}\xi\bigg|-C_{U}\delta^{n+1}2^{j\frac{n+1}{2}}\geq \frac{C_{L}}{2}\delta^{n}2^{j\frac{n+1}{2}}
	\end{align}
	if $|x_1-t|\leq \delta 2^{-j}$ and $|x^\prime|\leq 2^{-j/2}$.
	It then follows from \eqref{e3.15} and \eqref{e3.17} that
	\begin{align}\label{e3.18}
		\sup_{1\leq t\leq 2}|{\mathscr A}_t f|\geq \frac{C_{L}}{2}\delta^{n}2^{j\frac{n+1}{2}} -C_{\delta} 2^{-jN}2^{j\frac{n+1}{2}}\geq \frac{C_{L}}{4}\delta^{n}2^{j\frac{n+1}{2}},
	\end{align}
	when $1\leq x_1\leq 2$, $|x^\prime|\leq 2^{-j/2}$ and $j\geq \frac{1}{N}\log_{2}(\frac{4C_{\delta}}{\delta^{n}C_{L}}+1)$.

	Assume \eqref{e3.12} is true. Then from the definition of $f$ and \eqref{e3.18}, we have
	\begin{align}\label{e319}
		 \frac{C_{L}}{4}\delta^{n}2^{(n+1)j/2-(n-1)j/(2p)}
		&\leq  \big\|\sup_{1\leq t\leq 2}|{\mathscr A}_t f|\big\|_{L^p(\R^n)}            \\
		&\leq C\|f\|_{W^{s,p}(\Rn)}\leq C_\delta 2^{sj}2^{(n+1)j/2-(n+1)j/(2p)}.\notag
	\end{align}
	Let $j\to\infty$, then we obtain $s\geq 1/p$. This proves Lemma~\ref{le3.5}, and then the proof of  Proposition~\ref{prop3.4} is complete.
\end{proof}

We finally present the endgame in the

\begin{proof}[Proof of {\rm (i)} of   Theorem~\ref{th1.1}]
This is a consequence of Proposition~\ref{prop3.1} and Proposition~\ref{prop3.4}.
\end{proof}

\medskip


\medskip

\section{Proof of  {\rm (ii)}  of Theorem~\ref{th1.1}}
\setcounter{equation}{0}

In this section, we give a criterion that allows us to derive  $L^p$-boundedness for the maximal operator
$ {\frak M}^{\alpha}$  on ${\mathbb R^n}, n\geq 2$.  As a consequence,
   (ii) of   Theorem~\ref{th1.1} follows readily by applying
the result of Guth, Wang and Zhang \cite{GWZ} on local smoothing estimate on ${\mathbb R^2}$.
More precisely, we have the following result.

\begin{proposition}\label{prop4.1}
	Let $n\geq 2 $ and  $p>2$.   If the local smoothing estimate
	\begin{align}\label{e4.1}
		\big\| e^{it\sqrt{-\Delta}}f \big\|_{L^p(\Rn\times [1,2])} \leq C_{n,p} \|f\|_{W^{s, p}(\Rn)}
	\end{align}
	holds for some   $s\in\R$,   then   we have
	\begin{align}\label{e4.2}
		\big\|\sup_{t>0}|  {\frak M}_t^{\alpha}f |\big\|_{L^p(\Rn)}\leq C_{n,p,\alpha} \|f\|_{L^p(\Rn)}
	\end{align}
whenever
	$
	{\rm Re} \alpha > \max \big\{ -{(n-1)/p}, \ s -{(n-1)/2} +{1/p}\big\}.
	$
\end{proposition}

\medskip

The proof of  Proposition~\ref{prop4.1} is inspired by \cite{MYZ}.
Let  $\varphi$ and $\{\psi_j\}_j$ be functions defined in Section~\ref{s2}. We write
\begin{align}\label{e4.3}
	\widehat{{\frak M}_t^\alpha f}(\xi)
	&= \varphi(t|\xi|)\widehat{m_{\alpha}}(t\xi)\hat{f}(\xi)+\sum_{j\ge 1}\psi_j(t|\xi|)\widehat{m_{\alpha}}(t\xi)\hat{f}(\xi)              \notag \\
	&=:\widehat{{\frak M}_{0,t}^\alpha f}(\xi)+\sum_{j\ge 1}\widehat{{\frak M}_{j,t}^\alpha f}(\xi).
\end{align}
To prove Proposition~\ref{prop4.1}, the first  strategy    is to
show that if one modifies the definition so that for each operator ${\frak M}_{j,t}^\alpha$,
the supremum is taken over  $1\leq t\leq 2$,
  then the resulting maximal function is bounded on $L^p({\mathbb R^n})$.

\begin{lemma}\label{le4.2}
	Let $n\geq 2 $ and  $p>2$.  Under the assumption \eqref{e4.1} of Proposition~\ref{prop4.1},
	 there exists some $\delta>0$ such that
	\begin{align}\label{e4.4}
		\big\|\sup_{t\in [1,2]}|  {\frak M}_{j,t}^{\alpha}f |\big\|_{L^p(\Rn)}\leq C2^{-\delta j} \|f\|_{L^p(\Rn)}.
	\end{align}
if ${\rm Re}\, \alpha > \max \big\{ -{(n-1)/p}, s -{(n-1)/2} +{1/p}\big\}$.
\end{lemma}

\begin{proof}
	By \eqref{e2.4}, \eqref{e2.5} and \eqref{e2.7}, \eqref{e4.4} reduces to
	\begin{align}\label{e4.5}
		\big\|\sup_{t\in [1,2]}|  \mathscr{A}_{j,t}f |\big\|_{L^p(\Rn)}\leq C2^{[(n-1)/2+{\rm Re}\,\alpha-\delta]j} \|f\|_{L^p(\Rn)},
	\end{align}
	where $\widehat{\mathscr{A}_{j,t}f}(\xi)=\psi_j(t|\xi|)\widehat{\mathscr{A}_{t}f}(\xi)$ and $\mathscr{A}_t f$ is defined in \eqref{e2.10}. By \eqref{ee2.1}, we can write
	\begin{align*}
		{\mathscr A}_{j,t} f(x)= C\sum_{l=0}^{N-1}\int_{\Rn} \left( b_l e^{2\pi i(x\cdot\xi+ t|\xi|)} +  d_l e^{2\pi i(x\cdot\xi- t|\xi|)}  \right)|t\xi|^{-l}\psi_j(t|\xi|) \hat{f}(\xi) \,\text{d}\xi,
	\end{align*}
	which is a linear combination of
	$$
	T_{\ell,j}f(x,t)=:\int_{\Rn}e^{2\pi i(x\cdot\xi\pm t|\xi|)}|t\xi|^{-\ell}\psi_j(t|\xi|)\hat{f}(\xi)\,\text{d}\xi,\ \ \ \ell=0,1,\cdots,N-1.
	$$
	Hence, the proof of  \eqref{e4.5}   reduces to
	\begin{align}\label{e4.6}
		\big\|\sup_{t\in [1,2]}|  T_{0,j}f(\cdot,t) |\big\|_{L^p(\Rn)}\leq C2^{ [(n-1)/2+{\rm Re}\,\alpha-\delta]j} \|f\|_{L^p(\Rn)}, \ \ \ j\geq 1.
	\end{align}

	Now we apply Lemma~\ref{le2.2} to deal with \eqref{e4.6}. 
	First, it follows from \cite[Theorem 2, Chapter IX]{St1} that
	\begin{align}\label{e4.8}
		\|  T_{0,j}f(\cdot,1) \|_{L^p(\Rn)}\leq C 2^{(n-1)(1/2-1/p)j}\|f\|_{L^p(\Rn)}.
	\end{align}
	Next, we observe that for any $1\leq t\leq 2$ and $j\ge 1$, there holds
	$$
	\big|\partial_\xi^\beta \big(\psi_j(t|\xi|)\big)\big|\leq C(1+|\xi|)^{-|\beta|},
	$$
	where $\beta$ is any multi-index. So $\psi_j(t|\cdot|)\in S^0$ uniformly $1\leq t\leq 2$ and $j\ge 1$, hence
	\begin{align}\label{e4.7}	
		\int_{\Rn}\bigg| \int_{\Rn} e^{2\pi i(x\cdot\xi\pm t|\xi|)}\psi_j(t|\xi|)\hat{f}(\xi)\,\text{d}\xi \bigg|^p\,\text{d}x\leq C	 \int_{\Rn}\bigg| \int_{\Rn} e^{2\pi i(x\cdot\xi\pm t|\xi|)}\tilde{\psi}_j(\xi)\hat{f}(\xi)\,\text{d}\xi \bigg|^p\,\text{d}x,
	\end{align}
	where constant $C$ is independent of $t$ and $j$. Here  $\tilde{\psi}_j$ equals to 1 if $|\xi|\in [2^{j-2}M,2^{j+1}M]$ and vanishes if $|\xi|\notin [2^{j-3}M,2^{j+2}M]$, so that $\tilde{\psi}_j$ equals to 1  on the support of $\psi_j(t|\cdot|)$ when $1\leq t\leq 2$. Then we  apply our assumption \eqref{e4.1} on local smoothing estimate to  \eqref{e4.7} to obtain
	\begin{align*}
		\|  T_{0,j}f \|_{L^p(\Rn\times [1,2])}\leq C 2^{s j}\|f\|_{L^p(\Rn)},
	\end{align*}
	and by the same token, the operator
	$$
	\partial_t T_{0,j}(x,t) = \int_{\Rn}e^{2\pi i(x\cdot\xi\pm t|\xi|)}\big(\pm 2\pi i|\xi|\psi_j(t|\xi|)+|\xi|\psi_j^\prime(t|\xi|)\big)\hat{f}(\xi)\,\text{d}\xi.
	$$
	satisfies
	\begin{align*}
		\|  \partial_tT_{0,j}f \|_{L^p(\Rn\times [1,2])}\leq C 2^{(s+1)j}\|f\|_{L^p(\Rn)}.
	\end{align*}
	Thus, we use  Lemma~\ref{le2.2} to get
	\begin{align*}
		\big\|\sup_{t\in [1,2]}|  T_{0,j}f(\cdot,t) |\big\|_{L^p(\Rn)}
		\leq C \big(2^{(n-1)(1/2-1/p)j}+2^{(s+1/p)j}\big)\|f\|_{L^p(\Rn)}.
	\end{align*}
	When ${\rm Re}\, \alpha > \max \big\{ -{(n-1)/p}, \ s -{(n-1)/2} +{1/p}\big\}$, it is seen that estimate \eqref{e4.8} holds with $\delta= {\rm Re}\, \alpha-\max \big\{ -{(n-1)/p}, \ s -{(n-1)/2} +{1/p}\big\}$.
\end{proof}

\medskip

Finally, we can apply  Lemma~\ref{le4.2} to prove   Proposition~\ref{prop4.1}.

\begin{proof}[Proof of Proposition~\ref{prop4.1}]
	By \eqref{e4.3} and \eqref{e2.5}, \eqref{e4.2} reduces to
	\begin{align}\label{e4.12}
		\big\|\sup_{t>0}|  {\frak M}_{j,t}^{\alpha}f |\big\|_{L^p(\Rn)}\leq C2^{-\delta j} \|f\|_{L^p(\Rn)}
	\end{align}
	for some $\delta>0$. Since $\ell^p\subseteq\ell^\infty$, we have
	\begin{align}\label{e4.13}
		\big\|\sup_{t>0}|  {\frak M}_{j,t}^{\alpha}f |\big\|_{L^p(\Rn)}\leq \bigg(\sum_{k\in\Z} \big\|\sup_{t\in [2^k,2^{k+1}]}|  {\frak M}_{j,t}^{\alpha}f |\big\|_{L^p(\Rn)}^p\bigg)^{1/p}.
	\end{align}
	However, it follows from Lemma~\ref{le4.2} and a rescaling $t\to 2^{-k} t$ that
	\begin{align}\label{e4.14}
		\big\|\sup_{t\in [2^k,2^{k+1}]}|  {\frak M}_{j,t}^{\alpha}f |\big\|_{L^p(\Rn)}\leq C2^{-\delta j}\|f\|_{L^p(\Rn)}.
	\end{align}
 Then for $2^k\leq t\leq 2^{k+1}$, there must be $|\xi|\in [2^{j-k-2}M,2^{j-k+1}M]$. This tells us that  we can rewrite \eqref{e4.14} as
	$$
	\big\|\sup_{t\in [2^k,2^{k+1}]}|  {\frak M}_{j,t}^{\alpha}f |\big\|_{L^p(\Rn)}\leq C2^{-\delta j}\|P_{j-k}f\|_{L^p(\Rn)}.
	$$
	This, together with \eqref{e4.13}, implies
	\begin{align*}
		\big\|\sup_{t>0}|  {\frak M}_{j,t}^{\alpha}f |\big\|_{L^p(\Rn)}
		&\leq C2^{-\delta j}\bigg(\sum_{k\in\Z} \|P_{j-k}f\|_{L^p(\Rn)}^p\bigg)^{1/p} \\
		&=C2^{-\delta j}\bigg\|\bigg(\sum_{k\in\Z}|P_{j-k}f|^p\bigg)^{1/p}\bigg\|_{L^p(\Rn)}\\
		&\leq C2^{-\delta j}\bigg\|\bigg(\sum_{k\in\Z}|P_{j-k}f|^2\bigg)^{1/2}\bigg\|_{L^p(\Rn)}
	\end{align*}
	since $p>2$. By the Littlewood-Paley inequality,
	$$
	 \bigg\|\bigg(\sum_{k\in\Z}|P_{j-k}f|^2\bigg)^{1/2}\bigg\|_{L^p(\Rn)}\leq C\|f\|_{L^p(\Rn)}.
	$$
This   proves \eqref{e4.12}. Hence,   the proof of   Proposition~\ref{prop4.1} is complete.
\end{proof}

\medskip

\noindent
\begin{remark} i)
In  the dimension $n\geq 3$    Gao {\it et al.} \cite{GLMX} obtained improved local smoothing estimates for the wave equation, that is,
 \eqref{e4.1} holds with $s= (n-1)({1/2}-1/p)-\sigma$ for all $\sigma<2/p-1/2$  when
\begin{eqnarray*}
	p>
	\left\{
	\begin{array}{lll}
		{2(3n+5)\over 3n+1}, \ \ \  {\rm for}\   & n \ {\rm odd};  \\[4pt]
		{2(3n+6)\over 3n+2}, \ \ \  {\rm for}\   &n \ {\rm even}.
	\end{array}
	\right.
\end{eqnarray*}
Applying  Proposition~\ref{prop4.1}, we get  that
  \eqref{e1.3} holds if 	${\rm Re}\,\alpha> 	\alpha(p, n)$ where
\begin{eqnarray} \label{e4.22}  \hspace{1cm}
	\alpha(p, n)=
	\left\{
	\begin{array}{lll}
		\max\left\{-{n-1\over p}, -\frac{3}{8}(n-1)+\frac{5-n}{4p},\frac{4(n-1)}{(3n+5)(n+3)}-\frac{n^{2}-5}{(n+3)p}\right\}, \ \ \ & {\rm for}\    n \ {\rm odd};  \\[6pt]
		\max\left\{-{n-1\over p},-\frac{3n-2}{8}-\frac{n-6}{4p},-\frac{n-1}{n+4}-\frac{n^{2}+n-6}{(n+4)p}\right\}, \ \ \   & {\rm for}\    n \ {\rm even}.
	\end{array}
	\right.
\end{eqnarray}
The above range $\alpha$   in \eqref{e4.22} for $p>2$ is strictly wider than   \eqref{e1.7}. However, the range $p$ in  \eqref{e4.22}  is not optimal.
What happens when $n\geq 3$ (and $p>2$) remains open.

\medskip

ii)  Under the assumption \eqref{e4.1} of Proposition~\ref{prop4.1},
it follows by  \eqref{e4.4}  that  for $n\geq 2 $ and  $p>2$,
\begin{align*}
	\big\|\sup_{t\in [1,2]}|  {\frak M}_{t}^{\alpha}f |\big\|_{L^p(\Rn)}\leq C \|f\|_{L^p(\Rn)}
\end{align*}
provided that
 ${\rm Re}\, \alpha > \max \big\{ -{(n-1)/p}, s -{(n-1)/2} +{1/p}\big\}$.
It is interesting to describe the full  range of $(p, q)$ such that
  \begin{align*}
  	\big\|\sup_{t\in [1,2]}|  {\frak M}_{t}^{\alpha}f |\big\|_{L^q(\Rn)}\leq C \|f\|_{L^p(\Rn)}.
  \end{align*}
For   $\alpha=0$, we refer it to \cite{Lee, Sc1, Sc, SS} and the references therein.
\end{remark}

\medskip

\section{Appendix}
\setcounter{equation}{0}

For the  convenience of the reader, we give the proof of \eqref{ebb} on the complete asymptotic expansion of Bessel function $J_{\beta}$ provided that  ${\rm Re}\, \beta>-1/2$.
Recall that from \cite[(16), p.338]{St1}, the Bessel function $J_{\beta}$ is given by 
$$
J_{\beta}(r)=\frac{(r/2)^{\beta}}{\Gamma(\beta+1/2)\pi^{1/2}}\int_{-1}^{1}e^{irt}(1-t^{2})^{\beta-1/2}dt.
$$

\begin{proposition} \label{prop5.1}
Let ${\rm Re}\, \beta>-1/2$ and $r\geq 1$. We have the complete asymptotic expansion
\begin{align}\label{ee5.1}
	J_\beta(r)\sim  r^{-1/2}e^{ir} \sum_{j= 0}^{\infty}b_{j} r^{-j} + r^{-1/2}e^{-ir} \sum_{j= 0}^{\infty}d_j r^{-j}
\end{align}
for suitable coefficients $b_j$ and $d_j$, and  $b_{0},d_{0}\neq 0$.
\end{proposition}

\begin{proof}
Let $\varphi\in C_{c}^{\infty}(\R)$ satisfy $\varphi=1$ on $[-1/4, 1/4]$
 and $\supp \varphi\subseteq [-1/2, 1/2]$. Then we write
\begin{eqnarray}\label{eq5.1}
 \int_{-1}^{1} e^{irt}(1-t^{2})^{\beta-1/2}dt
 =:	J^{(1)}_\beta(r) +J^{(2)}_\beta(r) + E_{\beta} (r),
\end{eqnarray}
where
\begin{align*}
 J^{(1)}_\beta(r)&=  \int_{-1}^{1}e^{irt}(1-t^{2})^{\beta-1/2}\varphi(1-t)dt,\nonumber\\
J^{(2)}_\beta(r)&=\int_{-1}^{1}e^{irt}(1-t^{2})^{\beta-1/2}\varphi(1+t)dt,\nonumber\\
E_{\beta} (r)&=\int_{-1}^{1}e^{irt}(1-t^{2})^{\beta-1/2}\left(1-\varphi(1-t)-\varphi(1+t)\right)dt.
\end{align*}

By a change of variable,  we  apply (d) of 5.1  of \cite[ p.355]{St1} to obtain
\begin{eqnarray}\label{eq5.2}
J^{(1)}_\beta(r)
 &=&e^{ir}\int_{0}^{\infty}e^{-irs}s^{\beta-1/2}(2-s)^{\beta-1/2}\varphi(s)ds\nonumber\\
&\sim& e^{ir}\sum_{j=0}^{\infty}\tilde{b}_{j}r^{-1/2-\beta-j},
\end{eqnarray}
where \begin{align*}
\tilde{b}_{j}&=(-1)^{-1/2-\beta-j}i^{j+\beta+1}\frac{j!}{\Gamma(j+\beta+1)}\left(\frac{d}{ds}\right)^{j}\left((2-s)^{\beta-1/2}\varphi(s)\right)\bigg|_{s=0},
\end{align*}
and so
 $
	\tilde{b}_{j}
 =i^{-j-\beta}\frac{j!}{\Gamma(j+\beta+1)}\left(\frac{d}{ds}\right)^{j}\left((2-s)^{\beta-1/2}\right)\bigg|_{s=0}.
$
  A similar argument as above shows that
\begin{eqnarray}\label{eq5.3}
J^{(2)}_\beta(r) &=&e^{-ir}\int_{0}^{\infty}e^{irs}(2s-s^{2})^{\beta-1/2}\varphi(s)ds\nonumber\\
 &\sim& e^{-ir}\sum\limits_{j=0}^{\infty} \tilde{d}_{j}r^{-1/2-\beta-j},
\end{eqnarray}
where
$$
\tilde{d}_{j}
=i^{j+\beta+1}\frac{j!}{\Gamma(j+\beta+1)}\left(\frac{d}{ds}\right)^{j}\left((2-s)^{\beta-1/2}\right)\bigg|_{s=0}.
$$

Finally we estimate the term $E_{\beta} (r)$.   By integration by parts,  we have that for all $k\geq1$ and $N\geq1$,
$$
\Big|\left(\frac{d}{dr}\right)^{k}E_{\beta} (r)\Big|\leq C(1+r)^{-N} 
$$
for some constant  $C=C(\beta, k)$. 
  This, in combination  with \eqref{eq5.1}, \eqref{eq5.2} and \eqref{eq5.3},  implies \eqref{ee5.1} with 
  $$
  	b_j= \frac{(1/2)^{\beta}}{\Gamma(\beta+1/2)\pi^{1/2}}\tilde{b}_j,\quad d_j= \frac{(1/2)^{\beta}}{\Gamma(\beta+1/2)\pi^{1/2}}\tilde{d}_j.
  $$ 
  Obviously,   $b_0\not=0$ and $d_0\not=0$.
   The proof of
  Proposition~\ref{prop5.1} is complete.
\end{proof}

\bigskip

 \noindent
{\bf Acknowledgments.}
The authors thank L. Roncal kindly informing us  the preprint \cite{NRS}  in which sharp conditions
  for the  spherical maximal operator on radial functions were found. L. Yan thanks
  X.T. Duong for helpful  discussions.

The authors   were supported  by National Key R$\&$D Program of China 2022YFA1005700.  N.J. Liu was supported by China
Postdoctoral Science Foundation (No. 2022M723673). L. Song was supported by NNSF of China (No. 12071490).

\bigskip

 \bibliographystyle{amsplain}

 \end{document}